\documentclass{amsart}
\usepackage{preamble}
\begin{document}

\title{Smooth Combinatorial Cubes are IDP}
\author{Juliana Curtis}
\address{University of Washington}
\email{jacca@uw.edu}

\begin{abstract}
    Tadao Oda conjectured that every smooth polytope has the Integer Decomposition Property. In this paper, we give a new, polytopal proof of this for a subclass of polytopes: smooth combinatorial cubes of any dimension. 
\end{abstract}

\maketitle

\section{Introduction}

For lattice polytopes $P$ and $Q$, we say that $(P,Q)$ has the \textit{Integer Decomposition Property}, or that it \textit{is} \textit{IDP}, if every lattice point in the Minkowski sum $P + Q = \{p + q : p \in P, q \in Q\}$ can be written as the sum of a lattice point in $P$ and a lattice point in $Q$. For a single polytope $P$, we say that $P$ is \textit{IDP} when $(P, kP)$ is IDP for all positive integers $k$. IDP polytopes are directly related to Ehrhart theory, and are of great interest in commutative algebra and the study of toric varieties, as well as being of use in integer programming. In general, it is an open question to characterize when a polytope is IDP. 

While being IDP is a global property of a polytope, we are interested in its relationship to the more local notion of smoothness. First, a $d$-dimensional polytope is \textit{simple} if each vertex is contained in exactly $d$ edges (and so also exactly $d$ facets). We define the \textit{primitive edge directions} of a vertex to be the smallest lattice directions along all its incident edges, and then say that a $d$-dimensional polytope is \textit{smooth} if it is simple and if the primitive edge directions at each vertex form a basis for the integer lattice $\Z^d$. 

In 1997, Tadao Oda made the following conjecture, documented in \cite{oda2008problems}, which remains unproven.

\begin{conjecture} [Oda's Conjecture]
    All smooth polytopes are IDP.
\end{conjecture}

This problem is motivated by its relationship to the study of toric varieties, as there is a correspondence between smooth polytopes and ample divisors of smooth toric varieties. Suppose that $X_\Sigma$ is a smooth projective toric variety, and that $\mathcal{L}$ is an ample line bundle on it. Oda's Conjecture is equivalent to the statement that the embedding of $X_\Sigma$ given by $\mathcal{L}$ is projectively normal, or in other words, that the multiplication map
$$
    H^0 ( X_\Sigma, \mathcal{L}) \otimes ... \otimes H^0 ( X_\Sigma, \mathcal{L}) \to 
    H^0 ( X_\Sigma, \mathcal{L}^{\otimes k})
$$
is surjective for every positive integer $k$ (\cite{hibi2007mini}). 

The interest in the consequences of smoothness is not limited to the IDP, but includes stronger properties such as the existence of a unimodular covering or triangulation.  Indeed, there is a hierarchy of properties, cataloged in \cite{Haase_2021}, of which the IDP is the weakest. And despite attracting considerable interest, including as the subject of an Oberwolfach mini-workshop in 2007, Oda's conjecture remains open, even in three dimensions. As such, even partial or computational results relating to any such properties are of interest, as in \cite{gubeladze2012convex}, \cite{firla1999hilbert}, \cite{Lundman_2012}, \cite{Bogart_2015}, and \cite{haase2010generating}.
In particular, recent progress was made towards proving the conjecture in \cite{2019BecksmoothCS3D}, where Beck et al.\ showed that 3-dimensional, centrally symmetric, smooth polytopes are IDP. 

We define a $d$-dimensional \textit{combinatorial cube} to be a polytope whose face poset is isomorphic to the face poset of the unit cube, $[0,1]^d$. The following theorem is implied by results in the literature of toric varieties, namely \cite[Corollary~4.2]{ikeda09} paired with \cite[Corollary 4.4]{batyrev91}. Here, we give a new, polytopal proof. 

\begin{theorem} \label{TheoremTwoCubesMEquivIDP}
    Suppose that $C$ and $C'$ are smooth $d$-dimensional combinatorial cubes which are Minkowski equivalent. Then, $(C, C')$ is IDP.
\end{theorem} 

Then, the following is immediate.

\begin{corollary} \label{CorSmoothCubesAreIDP}
    Smooth combinatorial cubes of any dimension are IDP. 
\end{corollary}

\section{Preliminaries}

\subsection{Properties of smooth, IDP, and Minkowski equivalent polytopes} 

In this paper, we will explore the structure imposed on polytopes when they are smooth; in particular, we will be concerned with whether faces of polytopes are parallel. In general, every $k$-dimensional polytope $P$ lies in a $k$-dimensional affine subspace $\aff(P)$ which is parallel to a unique $k$-dimensional linear subspace, $\lin(P)$. Two polytopes $P$ and $Q$ are \textit{parallel} when $\lin(P) = \lin(Q)$. 

This additional structure evident in certain classes of smooth polytopes will allow us to consider the integer decomposition property. We will employ the following well-known facts.

\begin{proposition}[\cite{cox2011toric}, \cite{viet1997normal}, \cite{koelman1993generators}] \label{PropIDPBasics} 
    Basic IDP properties:
    \begin{enumerate} 
        \item [(a)] \label{PropIDPBasics_a} Let $P$ be a $d$-dimensional lattice polytope. Then, $(P, kP)$ is IDP for all integers $k \geq d-1$. 
        \item [(b)] All lattice polygons are IDP. 
    \end{enumerate}
\end{proposition} 

We also have the following useful characterization of being IDP, which we use repeatedly.

\begin{proposition} [IDP Equivalence] \label{PropIDPEquiv}
    Let $P$ and $Q$ be lattice polytopes. For a lattice point $a \in \Zd$, define
    $$
        R_a = P \cap (a + (-Q)).
    $$
    Then, $(P, Q)$ is IDP if and only if for all $a$, $R_a$ contains a lattice point whenever it is nonempty.
\end{proposition}
\begin{proof}
    Suppose first that $(P,Q)$ is IDP, and let $a \in \Zd$ be such that $R_a$ is nonempty. Let $y \in R_a$ and define $\tilde{q} = a - y$, so
    $$
        a = y + \tilde{q} \in P + Q.
    $$
    Since $a$ is a lattice point in $P + Q$ and $(P,Q)$ is IDP, there are lattice points $p \in P$ and $q \in Q$ such that $a = p + q$. Then, we see that 
    $$
        p = a - q \in a + (-Q),
    $$
    so $p$ is a lattice point in $P$ and $a + (-Q)$ and therefore in $R_a$.

    Conversely, suppose that for every lattice point $a$ such that $R_a$ is nonempty, $R_a$ contains a lattice point. Let $x$ be a lattice point in $P + Q$, so $x = \tilde{p} + \tilde{q}$
    for points $\tilde{p} \in P$ and $\tilde{q} \in Q$, not necessarily lattice points. Rearranging, 
    $$
        \tilde{p} = x - \tilde{q} \in P \cap (x + (-Q)) = R_x.
    $$
    Thus, $R_x$ is nonempty, so by assumption there is a lattice point $p \in P \cap (x + (-Q))$. Therefore there is a $q \in Q$ such that $p = x - q$, and since $p$ and $x$ are both lattice points, so is $q$. Thus, $x = p + q$, the sum of a lattice point in $P$ and a lattice point in $Q$, so $(P,Q)$ is IDP.
\end{proof}

In this paper, we rely heavily on the use of \textit{unimodular transformations}, which are linear maps $\Rd \to \Rd$ that send the lattice $\Z^d$ bijectively to itself. Equivalently, a linear transformation is unimodular if and only if the matrix representing it has determinant $\pm 1$ and integer entries. Importantly, for every lattice basis, there exists a unimodular transformation which sends it to the standard basis, and it follows that unimodular transformations preserve the IDP.

Every polytope $P$ is equipped with a \textit{normal fan}, $N(P)$, the collection of the normal cones of its faces. If a polytope $Q$ has the same normal fan as $P$, then $P$ and $Q$ are \textit{Minkowski equivalent}. We use the following characterization of this property.

\begin{proposition} \label{PropMEquivWhenAllFacetsParallel}
    Two polytopes are Minkowski equivalent if and only if there is an isomorphism between their face posets such that all corresponding facets are positively parallel to each other. In this case, the corresponding faces of each are also Minkowski equivalent.
\end{proposition}

It follows quickly that Minkowski equivalence preserves smoothness of lattice polytopes. There are many geometric implications of Minkowski equivalence; in particular, we have the following.

\begin{lemma} \label{LemmaMEquivDisjointHyperplane}
    Let $P$ and $Q$ be disjoint, full-dimensional polytopes in $\Rd$ such that $P$ and $-Q$ are Minkowski equivalent. Then, there is a hyperplane which separates them that is parallel to a facet $T$ of $P$. More precisely, $Q$ lies strictly on the opposite side of $\aff(T)$ than $P$.
\end{lemma}

\begin{proof}
    By the hyperplane separation theorem, there exist a normal vector $h \in \Rd$ and real numbers $a < b$ such that 
    $$
        \max_{x \in P} \langle x, h \rangle = a
        \hspace{1cm} \text{and} \hspace{1cm} 
        \min_{x \in Q} \langle x, h \rangle = b.
    $$
    Let $F$ be the face of $P$ maximizing $h$, and $N_F(P)$ its normal cone. Then $h \in N_F(P)$, so by Carathéodory's Theorem \cite{caratheodory1911} for cones, 
    $$
        h = \la_1 y_1 + ... + \la_k y_k,
    $$ 
    where each $y_i \in N_F(P)$ is an extreme ray, each $\la_i > 0$, and $\{y_1, ..., y_k \}$ is linearly independent. By definition of normal cones, $F$ simultaneously maximizes these directions in $P$, so for each $i$ we define $a_i = \max_{x \in P} \lan x, y_i \ran$ and then get that 
    \begin{align*}
        a   &= \max_{x \in P} \; \langle x, \la_1 y_1 + ... + \la_k y_k \rangle \\
            &= \la_1 \max_{x \in P} \; \langle x, y_1 \rangle + ... + \la_k \max_{x \in P} \; \langle x, y_k \rangle \\
            &= \la_1 a_1 + ... + \la_k a_k.
    \end{align*}

    Next, we recall that $P$ and $-Q$ are Minkowski equivalent, so $N(P) = N(-Q)$. Thus there is a face $-G$ of $-Q$ such that $N_{-G}(-Q) = N_F(P)$. As $h$ is a positive linear combination of the $y_i \in N_{-G}(-Q)$, similarly $-G$ simultaneously maximizes these directions in $-Q$, so for each $i$ let $b_i = \max_{x \in -Q} \lan x, y_i \ran$. Then, we similarly get that
    \begin{align*}
        b   &= \min_{x \in Q} \langle x, h \rangle \\
            &= - \max_{x \in -Q} \langle x, h \rangle \\
            &= - (\la_1 b_1 + ... + \la_k b_k).
    \end{align*}
    Then, since $a < b$, by rearranging we see
    $$
        \la_1(a_1 + b_1) + ... + \la_k(a_k + b_k) < 0.
    $$
    As the $\la_i$ are all positive, there is at least one $j$ such that $a_j + b_j < 0$. Thus, 
    $$
        \max_{x \in P} \langle x, y_j \rangle = a_j < - b_j = \min_{x \in Q} \langle x, y_j \rangle.
    $$
    Therefore, as $y_j$ is normal to a facet of $P$, a hyperplane parallel to a facet of $P$ separates $P$ and $Q$. In particular, letting $T$ be the facet of $P$ with normal vector $y_j$, we see that $Q$ lies strictly on the opposite side of $\aff(T)$ than $P$.
\end{proof}

\subsection{Combinatorial cubes} 

First, we consider the $d$-dimensional unit cube, $[0,1]^d$. Its faces are in bijection with the pairs of disjoint subsets of $[d]$, so letting $I$ and $J$ be two such sets, we define the corresponding face $F_I^J$ to be
$$
    F _I ^J := 
        \Bigg \{ 
            (x_1, x_2, ..., x_d) \in [0,1]^d : 
                x_k = 
                \medmath{\begin{cases} 
                    0 & \text{ if } k \in I \\ 
                    1 & \text{ if } k \in J
                \end{cases}}
        \Bigg \}
$$
and observe that it is $(d-(|I|+|J|))$-dimensional. Then, let $C$ be an arbitrary $d$-dimensional combinatorial cube. As there is a bijection between the face poset of $C$ and that of the unit cube, we reuse the above labeling for the faces of $C$. 

\begin{proposition}\label{prop:oldnotationcontainment}
Basic properties of faces of cubes.
    \begin{enumerate} 
        \item The face $F_{I_1} ^{J_1}$ contains the face $F _{I_2} ^{J_2}$ if and only if $I_1 \sbe I_2$ and $J_1 \sbe J_2$.

        \item The intersection of two faces $F _{I_1} ^{J_1}$ and $F _{I_2} ^{J_2}$ is the face $F_{I_1 \cup I_2}^{J_1 \cup J_2}$.  
    \end{enumerate}
\end{proposition}

However, we will use a more concise notation for $F_I^J$: we denote elements of $J$ with a bar, rather than using the superscript, so instead of writing $F_{\{y\}}^{\{x, z\}}$, we will subsequently write $F_{\bar x y \bar z}$.

\vspace{-2mm}

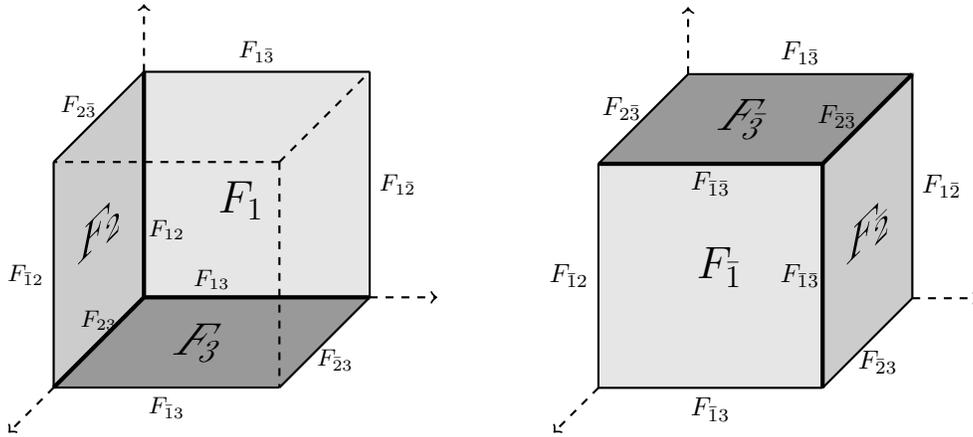
\begin{figure}[H]
    \begin{subfigure}[t]{0.38\textwidth}
        \resizebox{\linewidth}{!}{
            %%%%%%%%%%%%%%%%%%%%%%%%%%%%%%%%%%%%%%%%%%%%%%%%%%%%%%%%%%%%%%%%%%%%%%%%%%%%%%%%%%%%%%
%%%%% 2nd view of 3D cube
%%%%%%%%%%%%%%%%%%%%%%%%%%%%%%%%%%%%%%%%%%%%%%%%%%%%%%%%%%%%%%%%%%%%%%%%%%%%%%%%%%%%%%

\begin{tikzpicture}[thick,scale=3] \label{secondview3dcube}
    \coordinate (x1) at (0, 0);
    \coordinate (x13) at (0, 1);
    \coordinate (opp) at (1, 1);
    \coordinate (x12) at (1, 0);
    \coordinate (origin) at (0.4, 0.4);
    \coordinate (x3) at (0.4, 1.4);
    \coordinate (x23) at (1.4, 1.4);
    \coordinate (x2) at (1.4, 0.4);

    \coordinate (top_e3) at (0.4, 1.7);
    \coordinate (top_e2) at (1.7, 0.4);
    \coordinate (top_e1) at (-0.2, -0.2);

    \coordinate (F_12) at (0.5, 0.7);
    \node at (F_12){\small $F_{12}$};
    
    \coordinate (F_13) at (0.7, 0.48);
    \node at (F_13){\small $F_{13}$};

    \coordinate (F_23) at (0.2, 0.3);
    \node at (F_23){\small $F_{23}$};

    % lines from origin out
    \draw[ultra thick] (origin) to (x3);
    \draw[ultra thick] (origin) to (x2);
    \draw[ultra thick] (origin) to (x1);

    % lines opposite the above 3
    \draw[dashed] (opp) -- (x23);
    \draw[dashed] (opp) -- (x13);
    \draw[dashed] (opp) -- (x12);

    % lines hexagon outside of cube
    \draw[thick] (x12) to node[below] {\small $F_{\bar1 3}$} (x1) ;
    \draw[thick] (x1) to node[left] {\small $F_{\bar1 2}$} (x13);
    \draw[thick] (x13) to node[above=4, left=-3] {\small $F_{2 \bar3}$} (x3);
    \draw[thick] (x3) to node[above] {\small $F_{1 \bar3}$} (x23);
    \draw[thick] (x12) to node[right=4, below=-1] {\small $F_{\bar2 3}$} (x2);
    \draw[thick] (x2) to node[right] {\small $F_{1 \bar2}$} (x23);
    
    \draw[fill=black,opacity=0.4] (x1) -- (origin) -- (x2) -- (x12);
    \draw[fill=black,opacity=0.2] (x1) -- (x13) -- (x3) -- (origin);
    \draw[fill=black,opacity=0.1] (origin) -- (x3) -- (x23) -- (x2);

    \draw[->, dashed, thick]  (x3) -- (top_e3);
    \draw[->, dashed, thick]  (x2) -- (top_e2);
    \draw[->, dashed, thick]  (x1) to (top_e1);

    \node at (0.83,0.83) {\huge $F_1$ };
    \node[yslant = 1] at (0.2,0.69) {\huge $F_2$ };
    \node[xslant=0.53] at (0.65,0.2) {\huge $F_3$ };

\end{tikzpicture}
        }
    \end{subfigure}
    \hspace{0.05\textwidth}
    \begin{subfigure}[t]{0.38\textwidth}
        \resizebox{\linewidth}{!}{
            %%%%%%%%%%%%%%%%%%%%%%%%%%%%%%%%%%%%%%%%%%%%%%%%%%%%%%%%%%%%%%%%%%%%%%%%%%%%%%%%%%%%%%
%%%%% 1st view of 3D cube
%%%%%%%%%%%%%%%%%%%%%%%%%%%%%%%%%%%%%%%%%%%%%%%%%%%%%%%%%%%%%%%%%%%%%%%%%%%%%%%%%%%%%%

\begin{tikzpicture}[thick,scale=3] \label{firstview3dcube}
    \coordinate (x1) at (0, 0);
    \coordinate (x13) at (0, 1);
    \coordinate (opp) at (1, 1);
    \coordinate (x12) at (1, 0);
    \coordinate (origin) at (0.4, 0.4);
    \coordinate (x3) at (0.4, 1.4);
    \coordinate (x23) at (1.4, 1.4);
    \coordinate (x2) at (1.4, 0.4);

    \coordinate (top_e3) at (0.4, 1.7);
    \coordinate (top_e2) at (1.7, 0.4);
    \coordinate (top_e1) at (-0.2, -0.2);

    % lines opposite the above 3
    \draw[ultra thick] (opp) to node[above=1, left] {$F_{\bar2 \bar3}$} (x23);
    \draw[ultra thick] (opp) to node[below=-1] {$F_{\bar1 \bar3}$} (x13);
    \draw[ultra thick] (opp) to node[left=-3] {$F_{\bar1 \bar2}$} (x12);

    % lines hexagon outside of cube
    \draw[thick] (x12) to node[below] {$F_{\bar1 3}$} (x1) ;
    \draw[thick] (x1) to node[left] {$F_{\bar1 2}$} (x13);
    \draw[thick] (x13) to node[above=4, left=-3] {$F_{2 \bar3}$} (x3);
    \draw[thick] (x3) to node[above] {$F_{1 \bar3}$} (x23);
    \draw[thick] (x12) to node[right=4, below=-1] {$F_{\bar2 3}$} (x2);
    \draw[thick] (x2) to node[right] {$F_{1 \bar2}$} (x23);
    
    \draw[fill=black,opacity=0.1] (x1) -- (x13) -- (opp) -- (x12);
    \draw[fill=black,opacity=0.4] (x13) -- (x3) -- (x23) -- (opp);
    \draw[fill=black,opacity=0.2] (x12) -- (opp) -- (x23) -- (x2);

    \draw[->, dashed, thick]  (x3) -- (top_e3);
    \draw[->, dashed, thick]  (x2) -- (top_e2);
    \draw[->, dashed, thick]  (x1) to (top_e1);

    \node at (0.54,0.54) {\huge $F_{\bar1}$ };
    \node[yslant = 1] at (1.2,0.69) {\huge $F_{\bar2}$ };
    \node[xslant = 0.53] at (0.65,1.2) {\huge $F_{\bar3}$ };

\end{tikzpicture}
        }
    \end{subfigure}
    \caption{The 3-dimensional unit cube, viewed from the `inside' and `outside.'}
    \label{fig:standard3dcube}
\end{figure}

Consider the example shown in Figure \ref{fig:standard3dcube}. $C$ is 3-dimensional, and the six (2-dimensional) facets of $C$ are $F_1$, $F_{\bar{1}}$, $F_2$, $F_{\bar{2}}$, $F_3$, $F_{\bar3}$. As in Proposition \ref{prop:oldnotationcontainment}, $F_1$ contains the 1-dimensional faces $F_{12}, F_{13}, F_{1 \bar 2}, F_{1 \bar3}$, and the intersection of the 2-dimensional faces $F_1$ and  $F_{\bar 2}$ is the 1-dimensional face $F_{1 \bar2}$.

We say that two facets $F_x$ and $F_{\bar x}$ of $C$ are \textit{opposite} each other. Since faces of combinatorial cubes are themselves cubes, we have that $F_{x y}$ and $F_{x \bar y}$ are \textit{opposite} each other within $F_x$. Two facets can be parallel only if they are opposite, as otherwise the cube would collapse to a lower dimension.

When $C$ is smooth, the primitive edge directions at each vertex of $C$ span the integer lattice, and there always exists a unimodular transformation which sends these directions to the standard basis vectors. Thus, since unimodular transformations and translations preserve both the IDP and subspace parallelism, we may assume that one vertex of $C$ lies at the origin and has primitive edge directions along the coordinate axes, as in Figure \ref{fig:arbitrarycubewithorigin}.

\vspace{2mm}

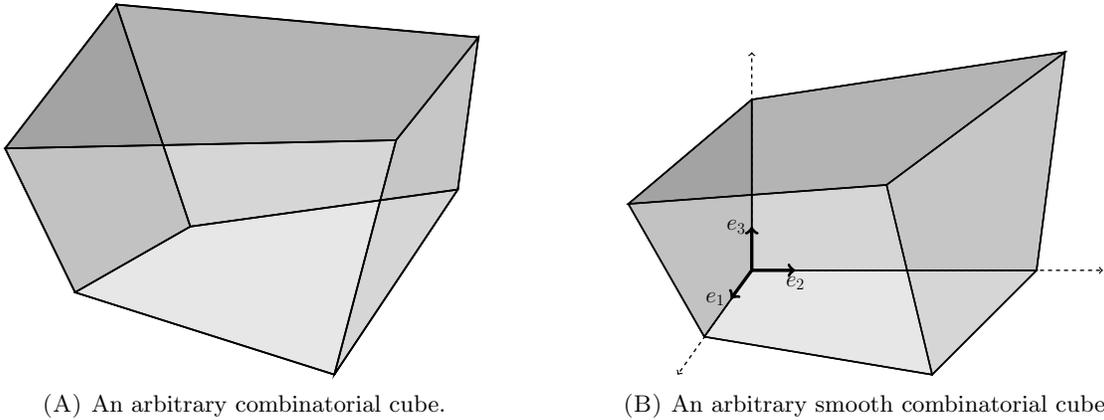
\begin{figure}[H]
    \begin{subfigure}[t]{0.38\textwidth}
        \resizebox{\linewidth}{!}{
            %%%%%%%%%%%%%%%%%%%%%%%%%%%%%%%%%%%%%%%%%%%%%%%%%%%%%%%%%%%%%%%%%%%%%%%%%%%%%%%%%%%%%%
%%%%% An arbitrary combinatorial cube
%%%%%%%%%%%%%%%%%%%%%%%%%%%%%%%%%%%%%%%%%%%%%%%%%%%%%%%%%%%%%%%%%%%%%%%%%%%%%%%%%%%%%%

\begin{tikzpicture}[] \label{arbitrarycube}

    % The 6 vertices of the cube
    \coordinate (orig) at (4.5, 2.6); 
    
    \coordinate (x1) at (1.7, 1);
    \coordinate (x2) at (11, 3.5);
    \coordinate (x3) at (2.7, 8);
    
    \coordinate (x13) at (0, 4.5);
    \coordinate (x123) at (9.5, 4.7);
    \coordinate (x12) at (8, -1);
    \coordinate (x23) at (11.5, 7.2);

    % Front face
    \draw[fill=black,opacity=0.05]     (x1) to (x12) to (x123) to (x13) to cycle;
    \draw[very thick]                  (x1) to (x12) to (x123) to (x13) to cycle;
    
    % Right face
    \draw[fill=black,opacity=0.12]     (x2) to (x23) to (x123) to (x12) to cycle;
    \draw[very thick]                  (x2) to  (x23) to  (x123) to  (x12) to  cycle;
    
    % Top face
    \draw[fill=black,opacity=0.2]      (x3) to (x13) to (x123) to (x23) to cycle;
    \draw[very thick]                  (x3) to (x13) to (x123) to (x23) to  cycle;

    % Bottom face
    \draw[fill=black,opacity=0.05]     (orig) to (x1) to (x12) to (x2) to cycle;
    \draw[very thick]                  (orig) to (x1) to (x12) to (x2) to cycle;
    
    % Back face
    \draw[fill=black,opacity=0.12]     (orig) to (x2) to (x23) to (x3) to cycle;
    \draw[very thick]                  (orig) to (x2) to (x23) to (x3) to cycle;

    % Left face
    \draw[fill=black,opacity=0.2]      (orig) to (x3) to (x13) to (x1) to cycle;
    \draw[very thick]                  (orig) to (x3) to (x13) to (x1) to cycle;

\end{tikzpicture}

% To color the faces, one coloring that works well is yellow, blue, red, green, purple, orange
        }
        \caption{An arbitrary combinatorial cube.}
        \label{fig:arbitrarycube}
    \end{subfigure}
    \hspace{0.1\textwidth}
    \begin{subfigure}[t]{0.42\textwidth}
        \resizebox{\linewidth}{!}{
            %%%%%%%%%%%%%%%%%%%%%%%%%%%%%%%%%%%%%%%%%%%%%%%%%%%%%%%%%%%%%%%%%%%%%%%%%%%%%%%%%%%%%%
%%%%% An abitrary SMOOTH combinatorial cube--after unimodular transformation 
%%%%%%%%%%%%%%%%%%%%%%%%%%%%%%%%%%%%%%%%%%%%%%%%%%%%%%%%%%%%%%%%%%%%%%%%%%%%%%%%%%%%%%

\begin{tikzpicture}[] \label{arbitrarycubewithorigin} 

    % The 6 vertices of the cube
    \coordinate (orig) at (3.25, 2.75);
    
    \coordinate (x1) at (2, 1);
    \coordinate (x2) at (10.75, 2.75);
    \coordinate (x3) at (3.25, 7.25);
    
    \coordinate (x13) at (0, 4.5);
    \coordinate (x123) at (6.8, 5);
    \coordinate (x12) at (8, 0);
    \coordinate (x23) at (11.5, 8.5);

    % The ends of the coordinate axes
    \coordinate (axis1) at (1.3, 0);
    \coordinate (axis2) at (12.5, 2.75);
    \coordinate (axis3) at (3.25, 8.5);

    % Draw the coordinate axes
    \draw[->, dashed, line width=1pt] (orig) to (axis1); 
    \draw[->, dashed, line width=1pt] (orig) to (axis2); 
    \draw[->, dashed, line width=1pt] (orig) to (axis3); 

    % Front face
    \draw[fill=black,opacity=0.05]     (x1) to (x12) to (x123) to (x13) to cycle;
    \draw[very thick]                  (x1) to (x12) to (x123) to (x13) to cycle;
    
    % Right face
    \draw[fill=black,opacity=0.12]     (x2) to (x23) to (x123) to (x12) to cycle;
    \draw[very thick]                  (x2) to  (x23) to  (x123) to  (x12) to  cycle;
    
    % Top face
    \draw[fill=black,opacity=0.2]      (x3) to (x13) to (x123) to (x23) to cycle;
    \draw[very thick]                  (x3) to (x13) to (x123) to (x23) to  cycle;

    % Bottom face
    \draw[fill=black,opacity=0.05]     (orig) to (x1) to (x12) to (x2) to cycle;
    \draw[very thick]                  (orig) to (x1) to (x12) to (x2) to cycle;
    
    % Back face
    \draw[fill=black,opacity=0.12]     (orig) to (x2) to (x23) to (x3) to cycle;
    \draw[very thick]                  (orig) to (x2) to (x23) to (x3) to cycle;

    % Left face
    \draw[fill=black,opacity=0.2]      (orig) to (x3) to (x13) to (x1) to cycle;
    \draw[very thick]                  (orig) to (x3) to (x13) to (x1) to cycle;

    % Primitive edge directions from the origin
    \coordinate (a1) at (2.7, 2);
    \coordinate (a2) at (4.4, 2.75);
    \coordinate (a3) at (3.25, 3.9);
    
    \draw[->, line width=2.5pt] (orig) to (a1) node[left] {\huge $e_1$};
    \draw[->, line width=2.5pt] (orig) to (a2) node[below] {\huge $e_2$};
    \draw[->, line width=2.5pt] (orig) to (a3) node[left] {\huge $e_3$};

\end{tikzpicture}
        }
        \caption{An arbitrary smooth combinatorial cube.}
        \label{fig:arbitrarycubewithorigin}
    \end{subfigure}
    \caption{Combinatorial cubes in dimension 3.}
    \label{fig:arbcubes}
\end{figure}

In particular, we call a face $F_I^J$ of $C$ \textit{primary} if $J = \es$, that is, when it lies entirely within a coordinate (linear) subspace, and so contains the origin. We note that a primary face $F_I$ lies in $\spann( \{e_k : k \in [d], k \notin I \})$.

\section{Parallelism in combinatorial cubes}

\subsection{Properties of parallel subspaces} 

We begin by collecting a few basic results of linear algebra which will be useful. 

\begin{lemma} \label{LemmaH1H2Parallel}
    Suppose that $H_1$ and $H_2$ are $k$-dimensional affine subspaces in $\Rd$ and that each contains two $(k-1)$-dimensional, non-parallel, affine subspaces: $F_1, G_1 \sbe H_1$ and $F_2, G_2 \sbe H_2$. Suppose also that $F_1$ is parallel to $F_2$ and $G_1$ is parallel to $G_2$, as pictured in Figure \ref{fig:lemma3_1}. Then, $H_1$ is parallel to $H_2$.
\end{lemma}

\begin{proof}
    Since $F_1$ is not parallel to $G_1$, for dimensional reasons, 
    $$
        \spann(\lin(F_i), \lin(G_i)) = \lin(H_i)
    $$
    for each $i$. Thus, $\lin(H_1) = \lin(H_2)$, so $H_1$ and $H_2$ are parallel.
\end{proof}

\begin{corollary} \label{CorDegenerateCase}
     Let $C$ be a $d$-dimensional smooth combinatorial cube for $d \geq 3$, and let $x,y,z \in [d]$ be distinct. Suppose that the pair of faces $(F_{x z},F_{x \bar z})$ are parallel and the pair of faces $(F_{y z},F_{y \bar z})$ are parallel. Then, $F_z$ is parallel to $F_{\bar z}$.
\end{corollary} 
\begin{proof}
    As $F_{x z}, F_{y z} \sbe F_z$ and $F_{x \bar z}, F_{y \bar z} \sbe F_{\bar z}$, by Lemma \ref{LemmaH1H2Parallel}, $F_z$ is parallel to $F_{\bar z}$.
\end{proof}

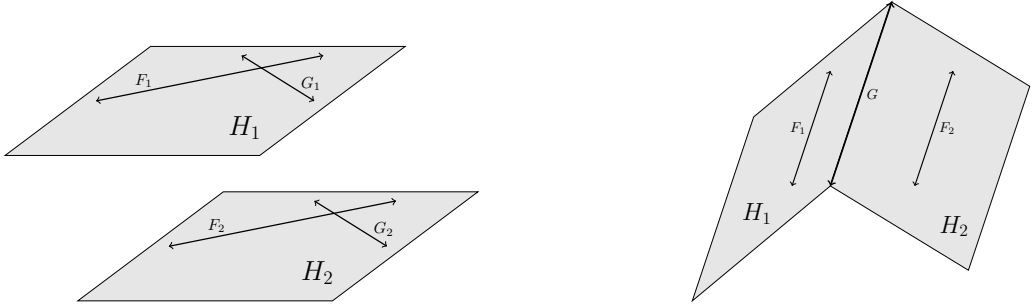
\begin{figure}[h]
    \begin{subfigure}[t]{0.45\textwidth}
        \resizebox{\linewidth}{!}{
            \begin{tikzpicture}[scale=0.65]
    
    % Define planes
    % H_1
    \coordinate (A1) at (0.25, 4);
    \coordinate (B1) at (7.25, 4);
    \coordinate (C1) at (11.25, 7);
    \coordinate (D1) at (4.25, 7);
    
    % H_2
    \coordinate (A2) at (2.25, 0);
    \coordinate (B2) at (9.25, 0);
    \coordinate (C2) at (13.25, 3);
    \coordinate (D2) at (6.25, 3);
    
    % Draw planes
    \draw[fill=black!10] (A1) node[right=120, above=7]{\huge $H_1$} -- (B1) -- (C1) -- (D1) -- cycle;
    \draw[fill=black!10] (A2) node[right=120, above=7]{\huge $H_2$} -- (B2) -- (C2) -- (D2) -- cycle;

    % Define lines
    \coordinate (F1a) at (2.75, 5.5);
    \coordinate (F1b) at (9, 6.75);
    \coordinate (G1a) at (6.75, 6.75);
    \coordinate (G1b) at (8.75, 5.5);
    
    \coordinate (F2a) at (4.75, 1.5);
    \coordinate (F2b) at (11, 2.75);
    \coordinate (G2a) at (8.75, 2.75);
    \coordinate (G2b) at (10.75, 1.5);
    
    % Draw lines
    \draw[<->, thick] (F1a) node[right=30, above=5]{$F_1$} to (F1b);
    \draw[<->, thick] (G1a) to (G1b) node[left=2,above=3]{$G_1$};

    \draw[<->, thick] (F2a) node[right=30, above=5]{$F_2$} to (F2b);
    \draw[<->, thick] (G2a) to (G2b)node[left=2,above=3]{$G_2$};
    
\end{tikzpicture}
        }
        \caption{Lemma \ref{LemmaH1H2Parallel}: In $\R^3$, $H_1$ and $H_2$ are planes while $F_1, F_2, G_1$, and $G_2$ are lines.}
        \label{fig:lemma3_1}
    \end{subfigure}
    \hspace{0.05\textwidth}
    \begin{subfigure}[t]{0.45\textwidth}
        \centering{
            \resizebox{0.72\linewidth}{!}{
                \begin{tikzpicture}[scale=0.55]

    % Define planes
    % H_1
    \coordinate (A1) at (0,0);
    \coordinate (B1) at (2, 6);
    \coordinate (C) at (6.5, 9.75);
    \coordinate (D) at (4.5, 3.75);
    
    % H_2
    \coordinate (A2) at (9, 1);
    \coordinate (B2) at (11, 7);
    % share C
    % share D

    % Define lines
    \coordinate (F1a) at (3.25, 3.75);
    \coordinate (F1b) at (4.5, 7.5);
    \coordinate (F2a) at (7.25, 3.75);
    \coordinate (F2b) at (8.5, 7.5);

    % Draw planes
    \draw[fill=black!10] (A1) node[right=35, above=38]{\huge $H_1$} -- (B1) -- (C) -- (D) -- cycle;
    \draw[fill=black!10] (A2) node[left=7, above=17]{\huge $H_2$} -- (B2) -- (C) -- (D) -- cycle;

    % Draw lines
    \draw[<->, thick] (F1a) to node[left]{$F_1$} (F1b);
    \draw[<->, thick] (F2a) to node[right]{$F_2$}(F2b);

    \draw[<->, very thick] (C) to node[right]{$G$}(D);
    
\end{tikzpicture}
            }
        }
        \caption{Lemma \ref{LemmaIntersectionParallel}: In $\R^3$, $H_1$ and $H_2$ are planes while $F_1$, $F_2$, and $G$ are lines.}
        \label{fig:lemma3_3}
    \end{subfigure}
    \caption{Parallelism in subspaces.}
\end{figure}

\begin{lemma} \label{LemmaIntersectionParallel}
    Let $H_1$ and $H_2$ be two non-parallel, $(d-1)$-dimensional affine hyperplanes in $\Rd$, with $d \geq 3$. Suppose that they contain $(d-2)$-dimensional affine subspaces $F_1$ and $F_2$, respectively, which are parallel to each other, as pictured in Figure \ref{fig:lemma3_3}. Then, the intersection of $H_1$ and $H_2$ is parallel to $F_1$ (and also $F_2$).
\end{lemma}

\begin{proof}
    Let $G = H_1 \cap H_2$, so we have that $F_1, G \sbe H_1$ and $F_2, G \sbe H_2$. Then by the contrapositive of Lemma \ref{LemmaH1H2Parallel}, because $H_1$ is not parallel to $H_2$, it must be that $G$ is parallel to $F_1$. 
\end{proof}

\begin{corollary} \label{CorThreeParallelImpliesFourth}
    Let $C$ be a $d$-dimensional smooth combinatorial cube with $d \geq 3$. Fix $x,y \in [d]$ with $x \neq y $, and consider the four $(d-2)$-dimensional faces $F_{xy}$, $F_{x \bar y}$, $F_{\bar x y}$, $F_{\bar x \bar y}$. If three of them are parallel to each other, then so is the fourth. 
\end{corollary}

\begin{proof}
    Suppose that of the four faces $F_{x y}$, $F_{x \bar{y}}$, $F_{\bar{x} y}$, $F_{\bar{x} \bar{y}}$, the first three are parallel. (An example is shown in Figure \ref{fig:parallelproof}\subref{subfig:straightforward}.) We see that the fourth face, $F_{\bar x \bar y}$, is precisely the intersection of the facets $F_{\bar x}$ and $F_{\bar y}$. But $F_{\bar x}$ contains $F_{\bar x y}$, and $F_{\bar y}$ contains $F_{x \bar y}$, which are parallel to each other. So, by Lemma \ref{LemmaIntersectionParallel}, $F_{\bar x \bar y}$ is also parallel to $F_{x \bar y}$ and $F_{\bar x y}$.
\end{proof}

\subsection{Parallelism of facets of combinatorial cubes} 

Now, we are ready to examine the parallelism of faces in combinatorial cubes. We begin with the following lemma, which can be found in the text \cite[Section~2.5]{fulton} as Exercise (a), but for which we also provide a simple proof using polytopes.

\begin{lemma} \label{lem:BaseCase2ParallelFaces}
    Every smooth, 2-dimensional combinatorial cube has two parallel facets.
\end{lemma}

\begin{proof}
    Suppose that $C$ is a 2-dimensional, smooth combinatorial cube. Via unimodular transformation, we take one corner of $C$ to be at the origin with primitive edge directions $(1, 0)$ and $(0, 1)$. Then, since $C$ is smooth, the primitive edge directions at each of the four vertices must span the integer lattice $\Z^2$, so they must be of the form given in Figure \ref{fig:parallelsides2d}, for some integers $m$ and $n$.

    \begin{figure}[H]
        \centering
            \resizebox{0.6\linewidth}{!}{
                %%%%%%%%%%%%%%%%%%%%%%%%%%%%%%%%%%%%%%%%%%%%%%%%%%%%%%%%%%%%%%%%%%%%%
%Quadrilateral for showing 2D cubes have parallel sides
%%%%%%%%%%%%%%%%%%%%%%%%%%%%%%%%%%%%%%%%%%%%%%%%%%%%%%%%%%%%%%%%%%%%%

\begin{tikzpicture}[thick,scale=3]

    % Ends of coordinate axes
    \coordinate (e1) at (2.45, 0);
    \coordinate (e2) at (0, 1.41);

    % Corners of the quadrilateral
    \coordinate (origin) at (0, 0);
    \coordinate (x1) at (1.5, 0);
    \coordinate (x2) at (0, 1);
    \coordinate (x12) at (2, 1.2);

    % Ends of primitive edge directions from each corner
    \coordinate (x1_a) at (1.23, 0);
    \coordinate (x1_b) at (1.607, 0.25);
    
    \coordinate (x2_a) at (0, 0.7);
    \coordinate (x2_b) at (0.3, 1.035);
    
    \coordinate (x12_a) at (1.89, 0.94);
    \coordinate (x12_b) at (1.7, 1.175);

    \coordinate (orig_a) at (0.3, 0);
    \coordinate (orig_b) at (0, 0.3);

    % Note to self: to have dots on vertices, either of these works:
    %\fill (origin) circle (0.8pt);
    %\node at (origin) [circle,fill,inner sep=1.5pt]{}; 

    % The corners of the quadrilateral look weird with the arrows if these points aren't filled
    \foreach \Point in {(origin), (x1), (x2), (x12)}{
        \node at \Point [circle,fill,inner sep=0.6pt]{};
    }

    % lines from origin out
    \draw[->, dashed] (origin) to (e1); 
    \draw[->, dashed] (origin) to (e2); 

    % edges of quadrilateral
    \draw[] (origin) to (x1);
    \draw[] (origin) to (x2);
    \draw[] (x1) to (x12);
    \draw[] (x2) to (x12);

    % primitive edge directions from each corner of quadrilateral
    \draw[->, ultra thick] (origin) to (orig_a) node[below] {$(1,0)$};
    \draw[->, ultra thick] (origin) to (orig_b) node[left] {$(0,1)$};

    \draw[->, ultra thick] (x1) to (x1_a) node[below] {$(-1,0)$};
    \draw[->, ultra thick] (x1) to (x1_b) node[right] {$(m,1)$};

    \draw[->, ultra thick] (x2) to (x2_a) node[left] {$(0, -1)$};
    \draw[->, ultra thick] (x2) to (x2_b) node[above] {$(1, n)$};

    \draw[->, ultra thick] (x12) to (x12_a) node[right] {$(-m,-1)$};
    \draw[->, ultra thick] (x12) to (x12_b) node[above] {$(-1,-n)$};

    % Fill in the quadrilateral
    \draw[fill=black,opacity=0.1] (origin) -- (x1) -- (x12) -- (x2);

\end{tikzpicture}
            }
        \caption{2-dimensional cube, with primitive edge directions at each vertex.}
        \label{fig:parallelsides2d}
    \end{figure}
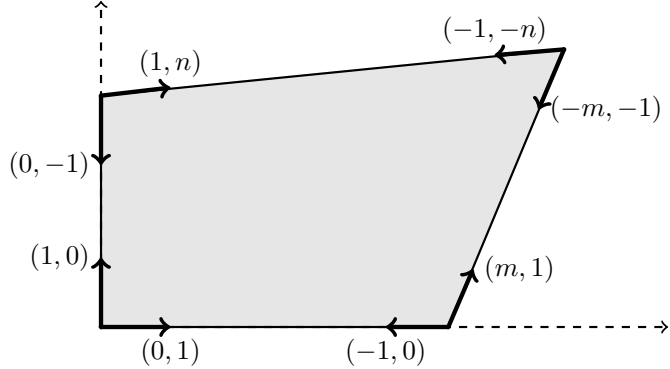

    Then, since $\{ (-1, -n), (-m, -1) \}$ must span $\Z^2$, it must be that 
    $$
        \det \begin{pmatrix}
                -1 & -m\\
                -n & -1
             \end{pmatrix}
        = \pm 1,
    $$
    which means $1 - mn = \pm 1$.

    If  $1-mn = -1$, then $nm = 2$, so the pair of values for $m$ and $n$ is either $\{1,2\}$ or $\{-1,-2\}$. If the former, then for either possible assignment of $m$ and $n$, the resulting top and right edges would not intersect. If the latter, $C$ would not be convex.
    Therefore, it must be that $1 - nm = 1$, i.e. $nm = 0$, so at least one of $n, m$ must be 0. In either case, $C$ has two parallel facets.
\end{proof}

Next, we prove the following lemma, which is a technical tool used in the main theorem.

\begin{lemma} \label{LemmaCaseA}
    Let $C$ be a smooth combinatorial cube of dimension $d \geq 3$ and suppose that all smooth combinatorial cubes of dimension smaller than $d$ have two parallel facets. Fix $x,y \in [d]$ with $x \neq y $, and suppose that the four  $(d-2)$-dimensional faces $F_{xy}, F_{x \bar y}, F_{\bar x y}, F_{\bar x \bar y}$ are all parallel to one another. Then, either $F_x$ and $F_{\bar x}$ are parallel or $F_y$ and $F_{\bar y}$ are parallel.
\end{lemma}

\begin{proof}
    Suppose without loss of generality that $x = 1$ and $y = 2$. Then, we proceed by induction on $d$. 
    
    When $d=3$, consider the facet $F_3$ of the cube. It has a pair of parallel faces by hypothesis, either $(F_{1 3}, F_{\bar1 3})$ or $(F_{2 3}, F_{\bar2 3})$. By Corollary \ref{CorDegenerateCase}, either $F_1$ and $F_{\bar 1}$ are parallel or $F_2$ and $F_{\bar 2}$ are parallel, respectively.

    Suppose $d > 3$, and that the result holds for all smaller-dimensional cubes. As in the base case, the facet of the cube $F_3$ has a pair of parallel faces $(F_{3 z}, F_{3 \bar z})$ for some $z \neq 3$. If $z = 1$ or $z = 2$, then by Corollary \ref{CorDegenerateCase}, either $F_1$ and $F_{\bar 1}$ are parallel or $F_2$ and $F_{\bar 2}$ are parallel, respectively. So, suppose that $z \neq 1,2,3$; without loss of generality, let $z = 4$.

    Since $F_{12}, F_{1 \bar 2}, F_{\bar 1 2}, F_{\bar 1 \bar 2}$ are all parallel, their intersections with $F_3$ are also parallel, i.e. $F_{123},$ $F_{1 \bar 2 3},$ $F_{\bar 1 2 3},$ $F_{\bar 1 \bar 2 3}$ are all parallel to each other. Then by the inductive hypothesis, $F_3$ has a pair of parallel faces, which is either $(F_{13}, F_{\bar1 3})$ or $(F_{2 3}, F_{\bar 2 3})$. Without loss of generality, assume the former are parallel. Then in $F_3$ there are pairs of parallel faces $(F_{13}, F_{\bar1 3})$ and $(F_{3 4}, F_{3 \bar4})$. It follows that $F_{1 3 4}$ and $F_{\bar1 3 \bar4}$ are parallel. Since $F_1$ contains both $F_{1 3 4}$ and $F_{1 2}$, and $F_{\bar1}$ contains both $F_{\bar1 3 \bar4}$ and $F_{\bar1 2}$, we get that 
    $$
        \lin(F_1) = \spann \big( \lin(F_{1 3 4}), \lin(F_{1 2}) \big) 
                  = \spann \big( \lin(F_{\bar1 3 \bar4}), \lin(F_{\bar1 2}) \big) = \lin(F_{\bar 1}).
    $$
    Thus, $F_1$ and $F_{\bar1}$ are parallel. 
\end{proof}

We can now proceed with the main theorem of this section.

\begin{figure}[h]
        \centering
        \begin{subfigure}[t]{0.45\textwidth}
            \resizebox{\linewidth}{!}{
                \begin{tikzpicture}[scale=0.7]
    \coordinate (orig) at (3.25, 2.75);
    
    \coordinate (x1) at (0, -1);
    \coordinate (x2) at (11.25, 2.75);
    \coordinate (x3) at (3.25, 7.25);
    
    \coordinate (x13) at (0.7, 4.5);
    \coordinate (x123) at (8.7, 4.5);
    \coordinate (x12) at (8, -1);
    \coordinate (x23) at (11.25, 7.25);

    \draw[fill=black,opacity=0.05] (x1) to (x12) to (x123) to (x13) to cycle; % F_{\bar1}
    \draw[fill=black,opacity=0.3] (x2) to (x23) to (x123) to (x12) to cycle; % F_{\bar2}
    \draw[fill=black,opacity=0.05] (x3) to (x13) to (x123) to (x23) to cycle; % F_{\bar3}

    \draw[fill=black,opacity=0.05] (orig) to (x1) to (x12) to (x2) to cycle; % F_3
    \draw[fill=black,opacity=0.05] (orig) to (x2) to (x23) to (x3) to cycle; % F_1
    \draw[fill=black,opacity=0.3] (orig) to (x3) to (x13) to (x1) to cycle; % F_2

    \draw[thick] (x1) to node[below]{\large$\bm{F_{\bar1 3}}$} 
        (x12) to 
        (x123) to 
        (x13) to % node[left]{\large$F_{\bar1 2}$} 
        cycle;
    \draw[thick] (x2) to %node[right]{\large$F_{1 \bar2}$} 
        (x23) to %node[above=3, left=1]{\large $F_{\bar2 \bar3}$} 
        (x123) to %node[below=7, left=-1]{\large $F_{\bar1 \bar2}$} 
        (x12) to %node[below=5, right=1]{\large$F_{\bar2 3}$} 
        cycle;
    \draw[thick] (x3) to %node[left=9, above=-1]{\large$F_{2 \bar3}$} 
        (x13) to %node[left=5, above=-1]{\large$F_{\bar1 \bar3}$} 
        (x123) to 
        (x23) to node[above]{\large $\bm{F_{1 \bar3}}$}
        cycle;

    \draw[thick] (orig) to %node[below=1, right=2]{\large$F_{2 3}$} 
        (x1) to 
        (x12) to 
        (x2) to 
        cycle;
    \draw[thick] (orig) to node[above]{\large$\bm{F_{13}}$} 
        (x2) to 
        (x23) to 
        (x3) to 
        cycle;
    \draw[thick] (orig) to %node[right=10, above=17]{\large$F_{12}$} 
        (x3) to 
        (x13) to 
        (x1) to 
        cycle;

    \draw[line width=1mm] (x3) to (x23);
    \draw[line width=1mm] (orig) to (x2);
    \draw[line width=1mm] (x1) to (x12);

\end{tikzpicture}
            }
            \caption{The faces $F_{1 \bar3}$, $F_{13}$, and $F_{\bar1 3}$ are parallel to each other.}
            \label{subfig:straightforward}
        \end{subfigure}
        \hfill
        \begin{subfigure}[t]{0.45\textwidth}
            \resizebox{\linewidth}{!}{
                \begin{tikzpicture}[scale=0.7]
    \coordinate (orig) at (3.25, 2.75);
    
    \coordinate (x1) at (0, -1);;
    \coordinate (x2) at (10.75, 2.75);
    \coordinate (x3) at (3.25, 7.25);
    
    \coordinate (x13) at (0.8, 4.7);
    \coordinate (x123) at (7, 4.3);
    \coordinate (x12) at (5.5, -1);
    \coordinate (x23) at (10.75, 9);

    \draw[fill=black,opacity=0.05] (x1) to (x12) to (x123) to (x13) to cycle; % F_{\bar1}
    \draw[fill=black,opacity=0.12] (x2) to (x23) to (x123) to (x12) to cycle; % F_{\bar2}
    \draw[fill=black,opacity=0.2] (x3) to (x13) to (x123) to (x23) to cycle; % F_{\bar3}

    \draw[fill=black,opacity=0.05] (orig) to (x1) to (x12) to (x2) to cycle; % F_3
    \draw[fill=black,opacity=0.12] (orig) to (x2) to (x23) to (x3) to cycle; % F_1
    \draw[fill=black,opacity=0.2] (orig) to (x3) to (x13) to (x1) to cycle; % F_2 

    \draw[very thick] (x1) to node[below]{\Large $\bm F_{\bar1 3}$} 
        (x12) to 
        (x123) to 
        (x13) to %node[left=10, above=1]{\large$F_{\bar1 2}$} 
        cycle;
    \draw[very thick] (x2) to node[right]{\Large$\bm F_{1 \bar2}$} 
        (x23) to %node[below=5]{\large $F_{\bar2 \bar3}$} 
        (x123) to %node[below=7, left=-1]{\large $F_{\bar1 \bar2}$} 
        (x12) to %node[below=5, right=1]{\large$F_{\bar2 3}$} 
        cycle;
    \draw[very thick] (x3) to node[left=9, above=-1]{\Large $\bm F_{2 \bar3}$} 
        (x13) to %node[right=18, above=-1]{\large$F_{\bar1 \bar3}$} 
        (x123) to 
        (x23) to %node[above]{\large$F_{1 \bar3}$} 
        cycle;

    \draw[very thick] (orig) to node[left = 5, above=3]{\Large $\bm F_{23}$} 
    (x1) to (x12) to (x2) to cycle;
    \draw[very thick] (orig) to node[right=40, above=1]{\Large $\bm F_{13}$} 
    (x2) to (x23) to (x3) to cycle;
    \draw[very thick] (orig) to node[right=11, above=14]{\Large $\bm F_{12}$} 
    (x3) to (x13) to (x1) to cycle;

    \draw[line width=1mm] (orig) to (x1);
    \draw[line width=1mm] (x1) to (x12);

    \draw[line width=1mm] (orig) to (x3);
    \draw[line width=1mm] (x3) to (x13);
    
    \draw[line width=1mm] (orig) to (x2);
    \draw[line width=1mm] (x2) to (x23);
    
\end{tikzpicture}
            }
            \caption{The pairs of faces $(F_{12}, F_{1 \bar2})$, $(F_{13}, F_{\bar1 3})$, and $(F_{23}, F_{2 \bar3})$ are parallel.}
            \label{subfig:pinwheel}
        \end{subfigure}

        \vspace{5mm}
        \begin{subfigure}[t]{0.45\textwidth}
            \resizebox{\linewidth}{!}{
                % Case (b)(ii)

\begin{tikzpicture}[scale=0.7]

    \coordinate (orig) at (3.25, 2.75);
    
    \coordinate (x1) at (0, -1);;
    \coordinate (x2) at (10.75, 2.75);
    \coordinate (x3) at (3.25, 7.25);
    
    \coordinate (x13) at (0.8, 4.7);
    \coordinate (x123) at (6.5, 5.5);
    \coordinate (x12) at (5.5, -1);
    \coordinate (x23) at (10.75, 9);

    \draw[fill=black,opacity=0.05] (x1) to (x12) to (x123) to (x13) to cycle; % F_{\bar1}
    \draw[fill=black,opacity=0.12] (x2) to (x23) to (x123) to (x12) to cycle; % F_{\bar2}
    \draw[fill=black,opacity=0.2] (x3) to (x13) to (x123) to (x23) to cycle; % F_{\bar3}

    \draw[fill=black,opacity=0.05] (orig) to (x1) to (x12) to (x2) to cycle; % F_3
    \draw[fill=black,opacity=0.12] (orig) to (x2) to (x23) to (x3) to cycle; % F_1
    \draw[fill=black,opacity=0.2] (orig) to (x3) to (x13) to (x1) to cycle; % F_2 

    \tikzset{decoration={snake,amplitude=.5mm,segment length=3mm,
                       post length=0mm,pre length=0mm}}
  
    \draw[line width=1mm] (x1) to node[below]                   {\large $F_{\bar1 3}$} (x12);
    \draw[decorate, very thick] (x13) to node[left=10, above=1] {\Large $\bm F_{\bar1 2}$} (x1);
    \draw[line width=1mm] (x2) to node[right]                   {\large $F_{1 \bar2}$} (x23);
    \draw[dotted, very thick] (x23) to node[below=5]            {\Large $\bm F_{\bar2 \bar3}$} (x123);
    \draw[decorate, very thick] (x123) to node[below=15, left=2]{\Large $\bm F_{\bar1 \bar2}$} (x12);
    \draw[dotted, very thick] (x12) to node[below=5, right=1]   {\Large $\bm F_{\bar2 3}$} (x2); 
    \draw[line width=1mm] (x3) to node[left=9, above=-1]        {\large $F_{2 \bar3}$} (x13);
    \draw[loosely dashed, very thick] (x13) to node[right=20, above=1]  {\Large $\bm F_{\bar1 \bar3}$} (x123);
    \draw[loosely dashed, very thick] (x23) to node[above]              {\Large $\bm F_{1 \bar3}$} (x3);

    \draw[line width=1mm] (orig) to node[left = 5, above=3]     {\large $F_{23}$} (x1); 
    \draw[line width=1mm] (orig) to node[right=40, above=1]     {\large $F_{13}$} (x2); 
    \draw[line width=1mm] (orig) to node[right=11, above=16]    {\large $F_{12}$} (x3); 

\end{tikzpicture}
            }
            \caption{The pairs of faces $(F_{\bar1 2}, F_{\bar1 \bar2})$, $(F_{1 \bar3}, F_{\bar1 \bar3})$, and $(F_{\bar2 3}, F_{\bar2 \bar3})$ are parallel, and $\lin(F_{\bar1 2}) \neq \spann(e_3)$.}
            \label{subfig:casebii}
        \end{subfigure}
        
        \caption{A 3-dimensional combinatorial cube, with various parallel faces.}
        \label{fig:parallelproof}
    \end{figure}
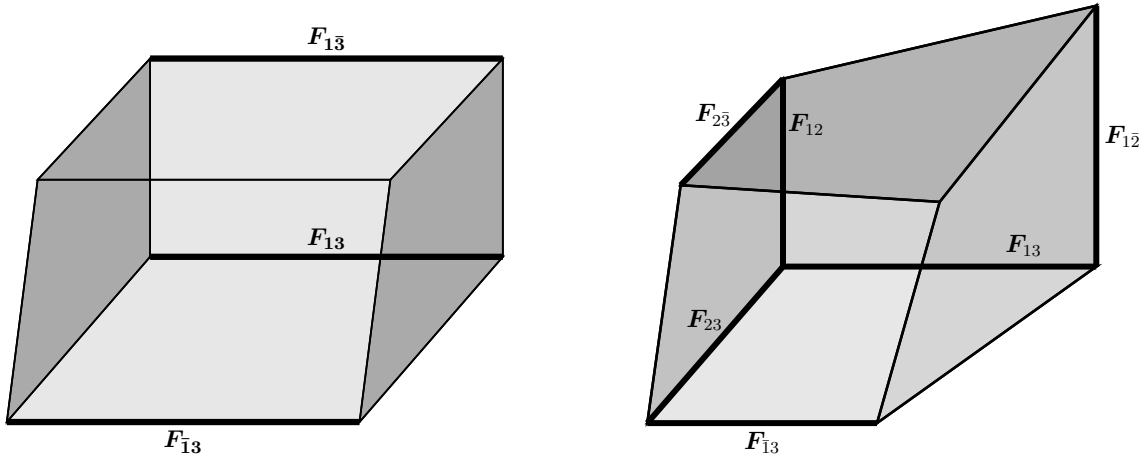
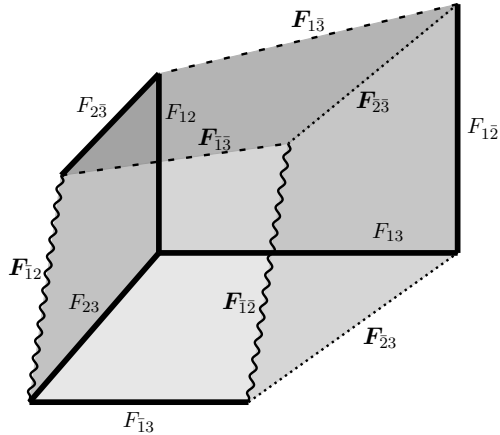

\begin{theorem} \label{TheoremMain2ParallelFaces}
    Let $C$ be a $d$-dimensional smooth combinatorial cube in $\Rd$ with $d \geq 2$. Then, $C$ has two parallel facets. 
\end{theorem}

\begin{proof}

    We proceed by induction on $d$. Lemma \ref{lem:BaseCase2ParallelFaces} gives us our base case when $d=2$, so let $d \geq 3$ and inductively assume that each smaller-dimensional such cube has two parallel facets. As usual, we may assume that one corner of $C$ lies at the origin and has primitive edge directions $e_1, e_2, ..., e_d$. 

    Consider the $d$ primary facets of the cube, $F_1, F_2, ..., F_d$, each of which is a $(d-1)$-dimensional cube itself. By our inductive hypothesis, each of them contains a pair of parallel $(d-2)$-faces. So, let us say that for each $i \in [d]$, $F_i$ contains a pair of parallel faces $(F_{i k_i}, F_{i \bar{k_i}})$ for some $k_i \in [d]$. 

    First, suppose that there are distinct indices $i,j \in [d]$ such that $k_i = k_j =: k$, and note that $k \neq i,j$. Then, by Corollary \ref{CorDegenerateCase}, it must be that $F_k$ and $F_{\bar k}$ are parallel.

    Then, we may assume that all of the $k_i$'s are distinct, as in Figure \ref{fig:parallelproof}\subref{subfig:pinwheel}. It follows that for each $k \in [d]$ there is a unique $i$ such that $k_i = k$. Then, consider the non-primary facets of the cube, $F_{\bar1}, F_{\bar2}, ..., F_{\bar d}$. By the inductive hypothesis, they also contain pairs of parallel $(d-2)$-faces. So, let us say that for each $i \in [d]$, $F_{\bar i}$ contains the pair of parallel faces $(F_{\bar i {m_i}}, F_{\bar i \bar{m_i}})$ for some $m_i \in [d]$. 

    We next consider different cases.

    \vspace{2mm} \noindent \textbf{Case 1}: There exists a $j$ such that the pair $(F_{j m_j}, F_{\bar j {m_j}})$ are parallel. 

    Without loss of generality, let us take $j = 1$ and $m_j = 2$. Then, the three faces $F_{1 2}$, $ F_{\bar1 2}$, and $F_{\bar1 \bar2}$ are all parallel, and by Corollary \ref{CorThreeParallelImpliesFourth}, $F_{1 \bar2}$ is parallel to them as well.  However, we see that this case is actually impossible if $d = 3$: The above tells us that $F_1$ has parallel pair $(F_{1 2}, F_{1 \bar2})$ and $F_2$ has parallel pair $(F_{1 2}, F_{\bar 1 2})$. So it is impossible for $k_1, k_2, k_3$ to be distinct, as the parallel faces in $F_3$ must either be $(F_{1 3}, F_{\bar1 3})$ or $(F_{2 3}, F_{\bar2 3})$. So, suppose that $d \geq 4$. Then, by Lemma \ref{LemmaCaseA}, either $F_1$ and $F_{\bar1}$ are parallel or $F_2$ and $F_{\bar2}$ are parallel.

    \vspace{2mm} \noindent \textbf{Case 2}: For every $i \in [d]$, $(F_{i {m_i}}, F_{\bar i {m_i}})$ are not parallel.  
    
    Consider $j=1$ and suppose without loss of generality that $m_j = 2$. Since $F_{1 2}$ and $F_{\bar1 2}$ are not parallel, it must be that $\lin(F_{\bar1 2}) \neq \spann(e_3, e_4, ..., e_d)$, as pictured in Figure \ref{subfig:casebii}. However,  
    $$
        \lin(F_{\bar1 \bar2}) = \lin(F_{\bar1 2}) \sbe \lin(F_2) = \spann(e_1, e_3, e_4, ..., e_d).
    $$
    By assumption, there is a $y$ such that the facet $F_y$ has the pair of parallel faces $(F_{2 y}, F_{\bar2 y})$. If $y=1$, then Lemma \ref{LemmaH1H2Parallel} gives us that the pair $(F_2, F_{\bar{2}})$ are parallel. Otherwise, without loss of generality we can take $y = 3$, and then $\lin(F_{\bar2 3}) = \lin(F_{2 3}) = \spann(e_1, e_4, e_5, ..., e_d)$.

    Since $\dim(F_{\bar1 \bar2}) = d-2$, we have that 
    $$
        \lin(F_{\bar1 \bar2}) = \spann(v_1, v_2, ..., v_{d-2})
    $$ 
    where each $v_k \in \spann(\{e_1, e_3, e_4, ..., e_d\})$. However, we observe two things: First, these vectors can't be such that $\lin(F_{\bar1 \bar2}) = \spann(e_3, e_4, ..., e_d)$ by assumption. Second, we can't have $\lin(F_{\bar1 \bar2}) = \spann( \{e_1, e_3, e_4, ..., e_d\} \sm \{e_p\})$ for any $p \in \{3, 4, ..., d\}$, otherwise $\lin(F_{\bar1 2})$ would equal $\lin(F_{p 2})$, collapsing the face $F_{2}$. So in particular, we know that there is at least one $v_k$ which can be written as a linear combination of $e_1, e_3, e_4, ..., e_d$ with a nonzero coefficient for $e_3$. Then,
    \begin{align*}
        \spann(\lin(F_{\bar1 \bar2}), \lin(F_{\bar2 3})) 
            &= \spann(e_1, e_4, e_5, ..., e_d, v_1, v_2, ..., v_{d-2}) \\
            &= \spann(e_1, e_3, e_4, ..., e_d) \\
            &= \lin(F_2).
    \end{align*}
    Lastly, since $F_{\bar1 \bar2}, F_{\bar2 3} \sbe F_{\bar2}$, we have that $\spann(\lin(F_{\bar1 \bar2}), \lin(F_{\bar2 3}))= \lin(F_{\bar2})$. So $\lin(F_2) = \lin(F_{\bar2})$, and thus $F_2$ and $F_{\bar2}$ are parallel.
\end{proof}

\section{IDP in prisms, prismatoids, and cubes}

A $d$-dimensional \textit{prism} is a polytope which is affinely equivalent to a polytope $Q \times [0,1]$ for some $(d-1)$-dimensional polytope $Q$. This yields \textit{top} and \textit{bottom} facets, $Q \times \{1\}$ and $Q \times \{0\}$, respectively, which are parallel to each other. We define a \textit{prismatoid} to be a lattice polytope whose face lattice is isomorphic to that of a prism and whose corresponding top and bottom facets are parallel (a slight restriction on the definition given in \cite{CS23}). By Theorem \ref{TheoremMain2ParallelFaces}, every $d$-dimensional smooth combinatorial cube is a prismatoid where $Q$ is a $(d-1)$-dimensional smooth combinatorial cube.

It is convenient to consider prismatoids whose top and bottom facets are parallel to the coordinate plane $\{(x_1, ..., x_d) : x_d = 0\}$. Again, since unimodular transformations and translations preserve the IDP, Minkowski equivalence, and subspace parallelism, in this paper we will assume all prismatoids have tops and bottoms parallel to this coordinate plane, and may further assume that smooth prismatoids have one vertex at the origin, with primitive edge directions $e_1, e_2, ..., e_d$.

\begin{lemma} \label{LemmaTopBottomComboPrismsMinkowskiEquiv}
    The top and bottom of a prismatoid are Minkowski equivalent. 
\end{lemma}

\begin{proof}
    Let $P$ be a prismatoid of dimension $d$ with top facet $T$ and bottom facet $B$. As $T$ and $B$ are combinatorially equivalent, there is an isomorphism between their face posets. So, letting two such corresponding facets of the top and bottom be $F_T$ and $F_B$, there is some third facet $F$ of $P$ such that $F_T = T \cap F$, and $F_B = B \cap F$. Then, since $T$ and $B$ are parallel,
    $$
        \lin(F_T) = \lin(F) \cap \lin(T) = \lin(F) \cap \lin(B) = \lin(F_B),
    $$
    so $F_T$ and $F_B$ are parallel, and moreover have outer normals pointing in the same direction. Therefore, every pair of corresponding facets of $T$ and $B$ are positively parallel, and it follows from Proposition \ref{PropMEquivWhenAllFacetsParallel} that $T$ and $B$ are Minkowski equivalent.
\end{proof}

Suppose that a prismatoid $P$ has bottom which lies in $\{(x_1, ..., x_d) : x_d = b\}$ and top which lies in $\{(x_1, ..., x_d) : x_d = b + h\}$ for integers $b$ and $h$. We define the \textit{slices} $S_l$ of $P$ as the nonempty intersections
$$
    S_l = P \cap \{(x_1, ..., x_d) : x_d = b + l\}
$$
for heights $l= 0,1, ..., h$, where $h$ is the \textit{height} of $P$.

\begin{lemma} \label{LemmaSlicesAreMEquiv}
    If $P$ is a smooth prismatoid of dimension $d$, then every slice of $P$ is a lattice polytope of dimension $d-1$ and is Minkowski equivalent to its bottom (and so also its top). 
\end{lemma}

\begin{proof}
    Let $P$ be a smooth prismatoid of dimension $d$ of height $h$ with top $T$ and bottom $B$. Let $n$ be the number of vertices of each. For $i \in [n]$, let $b_i$ be a vertex in $B$, $t_i$ be the corresponding vertex in $T$, and $E_i$ the edge connecting $b_i$ and $t_i$. Let the slices of $P$ be the $S_l$ as defined above.

    Since $P$ is smooth, the primitive edge direction of $E_i$ must be $(u_1, u_2, ..., u_{d-1}, 1) \in \Z^d$. Then, for each integer height $l \in [h]$, the vertices of the slice $S_l$ are precisely the intersections of the $E_i$ with $S_l$. Moreover, each $E_i$ intersects $S_l$ at the integer point $b_i + l(u_1, u_2, ..., u_{d-1}, 1)$, so $S_l$ has only integer vertices, and as no two such points can coincide, the convex hull of $S_l$ and $B$ is also a prismatoid. Therefore, for each $l$, $S_l$ is Minkowski equivalent to $B$ (and so also $T$) by Lemma \ref{LemmaTopBottomComboPrismsMinkowskiEquiv}.
\end{proof}

A useful two-dimensional result that does not require smoothness was proved by Haase et al.

\begin{theorem}[\cite{Haase08}] \label{TheoremMEquivPolygonsAreIDP}
     Let $P$ and $P'$ be Minkowski equivalent lattice polygons. Then, $(P, P')$ is IDP.
\end{theorem}

The analogous statement for dimensions higher than two is not true; as a counterexample, consider a pair $(P, kP)$ when $P$ is not IDP. The following lemma is a special case which we use to prove our main results.

\begin{lemma} \label{LemmaRecursiveIDP}
    Let $P$ and $P'$ be $d$-dimensional smooth prismatoids which are Minkowski equivalent. Let $S_l$ and $S_m'$ be the slices of $P$ and $P'$ respectively, and suppose that for every $l$ and $m$, the pair $(S_l, S_m')$ is IDP. Then, $(P, P')$ is IDP.
\end{lemma}

\begin{proof}
    Assume that the bottom $B$ of $P$ lies in $\{(x_1, ..., x_d) : x_d = 0\}$ and its top $T$ lies in $\{(x_1, ..., x_d) : x_d = t\}$ for some integer $t$, while the bottom $B'$ of $P'$ lies in $\{(x_1, ..., x_d) : x_d = b'\}$ and its top $T'$ lies in $\{(x_1, ..., x_d) : x_d = t'\}$ for integers $b' < t'$.
    
    To show that $(P, P')$ is IDP, we will use Proposition \ref{PropIDPEquiv}. Suppose that $R_a = P \cap (a + (-P'))$ is nonempty for some point $a = (a_1, ..., a_d) \in \Z^d$. We see that $a + (-P')$ has top $a + (-B')$ which is contained in $\{(x_1, ..., x_d): x_d = a_d - b' \}$ and bottom $a + (-T')$ which is contained in $\{(x_1, ..., x_d): x_d = a_d -t' \}$. 
    In order for $R_a$ to be nonempty, the top or bottom of one of $P$ or $a+(-P')$ lies in a hyperplane which lies between the top and bottom hyperplanes of the other. 
    So, suppose without loss of generality that $B$ lies between the top and bottom hyperplanes of $a + (-P')$. This means that there is a slice $S_m'$ of $P'$ such that the slice $G := a +(-S_m')$ of $a + (-P')$ lies in $\{(x_1, ..., x_d): x_d = 0\}$ along with $B$. 
    
    Suppose towards contradiction that $B$ does not intersect $G$. Since $B$ and $-G$ are Minkowski equivalent and full-dimensional within $\{(x_1, ..., x_d): x_d = 0\}$, Lemma \ref{LemmaMEquivDisjointHyperplane} gives that there is a $(d-2)$-face $\Tilde{B}$ of $B$ such that $G$ lies strictly on the opposite side of $\aff(\Tilde{B})$ from $B$.

    Since $P$ is a prismatoid, $P$ has a facet $F$ such that $\Tilde{B} = B \cap F$.
    As $P'$ is Minkowski equivalent to $P$, Proposition \ref{PropMEquivWhenAllFacetsParallel} gives us that $P'$ has a facet $F'$ positively parallel to $F$, that is, the two facets have outer normals pointing in the same direction. $G$ also has a face $\Tilde{G} = G \cap (a + (-F'))$ which is parallel to $\Tilde{B}$. Additionally, $F$ and $a + (-F')$ are parallel but with opposite facing outer normals.
    As the former contains $\Tilde{B}$ and the latter contains $\Tilde{G}$, it must be that $a + (-P')$ lies strictly on the opposite side of $\aff(F)$ than $P$, contradicting our supposition that $R_a$ is nonempty. Therefore, it must be that $B$ intersects $G$, as in Figure \ref{fig:slices_overlap}.

    \vspace{-4mm}
    \begin{figure}[H]
        \centering
            \resizebox{0.65\linewidth}{!}{
                \begin{tikzpicture}[scale=0.6]

    \coordinate (0) at (2, 2) {};
    \coordinate (1) at (6, 2) {};
    \coordinate (2) at (10, 6) {};
    \coordinate (3) at (4, 7) {};
    \coordinate (4) at (0, 5) {};
    \coordinate (5) at (11.75, 11) {};
    \coordinate (6) at (17.75, 11) {};
    \coordinate (7) at (5.75, 5) {};
    \coordinate (8) at (14.75, 3.5) {};
    \coordinate (9) at (20.75, 6.5) {};
    \coordinate (10) at (2.25, -0.6) {};
    \coordinate (11) at (6, -0.6) {};
    \coordinate (12) at (9.4, 2.6) {};
    \coordinate (13) at (4, 3.75) {};
    \coordinate (14) at (0.6, 1.9) {};
    \coordinate (15) at (12.15, 14) {};
    \coordinate (16) at (17.3, 14) {};
    \coordinate (17) at (6.3, 8.25) {};
    \coordinate (18) at (14.75, 7) {};
    \coordinate (19) at (20.1, 10) {};

    % \node at (0) [above] {$0$};
    % \node at (1) [above] {$1$};
    % \node at (2) [above] {$2$};
    % \node at (3) [above] {$3$};
    % \node at (4) [above] {$4$};
    % \node at (5) [above] {$5$};
    % \node at (6) [above] {$6$};
    % \node at (7) [above] {$7$};
    % \node at (8) [above] {$8$};
    % \node at (9) [above] {$9$};
    % \node at (10) [above] {$10$};
    % \node at (11) [above] {$11$};
    % \node at (12) [above] {$12$};
    % \node at (13) [above] {$13$};
    % \node at (14) [above] {$14$};
    % \node at (15) [above] {$15$};
    % \node at (16) [above] {$16$};
    % \node at (17) [above] {$17$};
    % \node at (18) [above] {$18$};
    % \node at (19) [above] {$19$};

    \coordinate (z0) at (12.4, 16.5) {};
    \coordinate (z1) at (17, 16.5) {};
    \coordinate (z2) at (6.7, 10.75) {};
    \coordinate (z3) at (14.75, 9.5) {};
    \coordinate (z4) at (19.6, 12.5) {};

    \coordinate (x0) at (1.7, 4.5) {};
    \coordinate (x1) at (6, 4.1) {};
    \coordinate (x2) at (10.5, 8.9) {};
    \coordinate (x3) at (4, 10.2) {};
    \coordinate (x4) at (-0.5, 8) {};

    \coordinate (y0) at (2.3, -1.1) {};
    \coordinate (y1) at (6, -1.1) {};
    \coordinate (y2) at (9.4, 2.2) {};
    \coordinate (y3) at (4, 3.05) {};
    \coordinate (y4) at (0.7, 0.9) {};

    % \node at (x0) [above] {$x_0$};
    % \node at (x1) [above] {$x_1$};
    % \node at (x2) [above] {$x_2$};
    % \node at (x3) [above] {$x_3$};
    % \node at (x4) [above] {$x_4$};
    
    % \node at (y0) [below] {$y_0$};
    % \node at (y1) [below] {$y_1$};
    % \node at (y2) [below] {$y_2$};
    % \node at (y3) [below] {$y_3$};
    % \node at (y4) [below] {$y_4$};

    % Flat edges
    % Smaller slice
    \draw[very thick] (3) to (4);
    \draw[very thick] (4) to (0);
    \draw[very thick] (0) to (1);
    \draw[very thick] (1) to (2);
    \draw[very thick] (2) to (3);

    \draw (13) to (14);
    \draw (14) to (10);
    \draw (10) to (11);
    \draw (11) to (12);
    \draw (12) to (13);
    
    % Larger slice
    \draw[very thick] (5) to (6);
    \draw[very thick] (5) to (7);
    \draw[very thick] (7) to (8);
    \draw[very thick] (6) to (9);
    \draw[very thick] (9) to (8);

    \draw (15) to (16);
    \draw (15) to (17);
    \draw (17) to (18);
    \draw (16) to (19);
    \draw (19) to (18);

    % Semi-vertical edges
    % Smaller slice
    % \draw (14) to (4);
    % \draw (10) to (0);
    % \draw (11) to (1);
    % \draw (12) to (2);
    % \draw (13) to (3);

    % Larger slice
    % \draw[dashed] (17) to (7);
    % \draw[dashed] (18) to (8);
    % \draw[dashed] (19) to (9);
    % \draw[dashed] (16) to (6);
    % \draw[dashed] (15) to (5);

    \draw[dashed] (7) to (z2);
    \draw[dashed] (8) to (z3);
    \draw[dashed] (9) to (z4);
    \draw[dashed] (6) to (z1);
    \draw[dashed] (5) to (z0);

    \draw[fill=black,opacity=0.15] (3) to (4) to (0) to (1) to (2) to cycle;
    \draw[fill=black,opacity=0.15] (6) to (5) to (7) to (8) to (9) to cycle;

    \draw (1.2, 9) node{\huge $\bm{a + (-P')}$};
    \draw (17.5, 3.6) node{\huge $\bm P$};

    \draw (2, 5.2) node{\huge $G$};
    \draw (17, 6) node{\huge $B$};

    \draw (8, 5.5) node{\Large $B \cap G$};

    % dotted lines
    \draw[dashed] (y4) to (x4);
    \draw[dashed] (y0) to (x0);
    \draw[dashed] (y1) to (x1);
    \draw[dashed] (y2) to (x2);
    \draw[dashed] (y3) to (x3);

\end{tikzpicture}
            }
        \caption{Slices $B$ of $P$ and $G$ of $a + (-P')$.}
        \label{fig:slices_overlap}
    \end{figure}
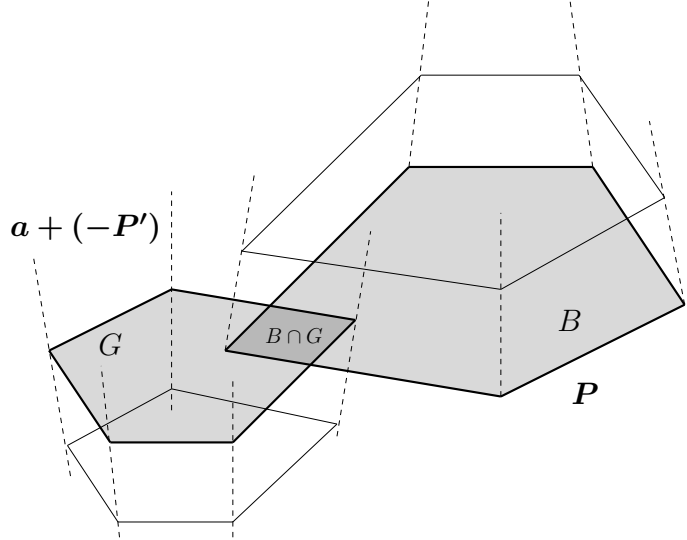

    By assumption, $(B, S_m')$ is IDP. By Proposition \ref{PropIDPEquiv}, since $B \cap G = B \cap (a +(-S_m'))$ is nonempty, it must contain a lattice point. Thus, $R_a$ contains a lattice point, so we conclude that $(P, P')$ is IDP.
\end{proof}

Using the above lemma and Theorem \ref{TheoremMEquivPolygonsAreIDP}, the following holds immediately.

\begin{theorem} \label{Theorem3DMEquivPseudoPrismsAreIDP}
    Let $P$ and $P'$ be Minkowski equivalent, smooth, 3-dimensional prismatoids. Then, $(P, P')$ is IDP.
\end{theorem}

It was previously proved in \cite{CS23} that every 3-dimensional smooth prismatoid has a unimodular cover, which implies that it is IDP; this also follows immediately from Theorem \ref{Theorem3DMEquivPseudoPrismsAreIDP}.

We can now prove our main result, Theorem \ref{TheoremTwoCubesMEquivIDP}.

\begin{proof}[{Proof of Theorem \ref{TheoremTwoCubesMEquivIDP}}]
    We proceed by induction on $d$. When $d=1$ it is trivial, and when $d = 2$, Theorem \ref{TheoremMEquivPolygonsAreIDP} provides the result, so suppose that $d > 2$ and that all pairs of Minkowski equivalent smooth $(d-1)$-cubes are IDP.
    
    By Theorem \ref{TheoremMain2ParallelFaces}, $C$ and $C'$ are smooth prismatoids. Then by Lemma \ref{LemmaSlicesAreMEquiv}, all slices of $C$ are Minkowski equivalent $(d-1)$-dimensional smooth cubes, and the same holds for $C'$. As $C$ and $C'$ are Minkowski equivalent to each other, it follows that every slice of $C$ is Minkowski equivalent to every slice of $C'$. Then by our inductive assumption, it must be that for every slice $S_l$ of $C$ and $S_m'$ of $C'$, the pair $(S_l, S_m')$ is IDP. Therefore, by Lemma \ref{LemmaRecursiveIDP}, $(C, C')$ is IDP.
\end{proof}

\section*{Acknowledgements}

The author is very grateful to Gaku Liu for his guidance and support during the research and writing of this paper. The author would also like to thank Chayim Lowen and Tuong Le for their enthusiasm and helpful feedback, and Sharon Robins for pointing out references in the literature. 

\vspace{3mm}
\noindent \textit{AI disclosure}: No AI was used in any capacity for the first version of this paper. Before the second version, Anthropic's Claude (Fable 5 and 5.1) was used to proofread, resulting in minor fixes to both grammar and mathematical consistency. The author takes full responsibility for the content of the paper.

\vspace{0.4cm}

\printbibliography

\end{document}